\documentclass[a4paper,12pt]{article}
\usepackage{amsmath, amsthm, amstext}
\usepackage{amssymb, array, amsfonts}
\usepackage[utf8]{inputenc}
\usepackage{graphicx}

\numberwithin{equation}{section}

\def \al{\alpha}

\def \ga{\gamma}

\def \ze{\zeta}

\def \si{\sigma}
\def \te{\theta}
\def \ph{\varphi}

\def \D{\Delta}

\def \O{\Omega}

\def \N{\mathbb{N}}
\def \R{\mathbb{R}}
\def \Z{\mathbb{Z}}
\def \T{\mathbb{T}}
\def\n{\nabla}
\def\dd{\partial}
\def\div{\operatorname{div}}

\def\1{1\!\!\!\!1}
\def\const{\operatorname{const}}

\def\dom{\operatorname{Dom}}

\def\mes{\operatorname{mes}}

\def\im{\operatorname{Im}}
\def\re{\operatorname{Re}}

\theoremstyle{plain}
\newtheorem{theorem}{\bf Theorem}[section]
\newtheorem{lemma}[theorem]{\bf Lemma}

\newtheorem{cor}[theorem]{\bf Corollary}

\theoremstyle{definition}

\theoremstyle{remark}
\newtheorem{rem}[theorem]{\bf Remark}

\renewcommand{\le}{\leqslant}
\renewcommand{\ge}{\geqslant}

\renewcommand{\qed}{\vrule height7pt width5pt depth0pt}

\title{On the Landis conjecture in a cylinder}

\author{N.~D.~Filonov, S.~T.~Krymskii
\thanks{This research is supported by the grant of Russian Science Foundation
No 22-11-00092, https://rscf.ru/project/22-11-00092/}}
\date{}
\begin{document}
\maketitle

\begin{abstract}
The equation $- \Delta u + V u = 0$ in the cylinder
$\R \times (0,2\pi)^d$ with periodic boundary conditions is considered.
The potential $V$ is assumed to be bounded, and both functions $u$ and $V$
are assumed to be {\it real-valued}.
It is shown that the fastest rate of decay at infinity of non-trivial solution $u$
is $O\left(e^{-c|w|}\right)$ for $d=1$ or $2$,
and $O\left(e^{-c|w|^{4/3}}\right)$ for $d\ge 3$.
Here $w$ is the axial variable.
\end{abstract}

\section*{Introduction}
\subsection{Formulation of the result}
Consider the equation
\begin{equation}
\label{01}
-\D u + V u = 0,
\end{equation}
where $V$ is a bounded measurable function. 
Landis has proposed the question about the fastest speed at which a solution $u$ 
can decrease at infinity if we consider the equation \eqref{01} in the entire space $\mathbb{R}^n$ 
or outside a ball in $\mathbb{R}^n$. 
There are two settings of the question: the real case and the complex case. 
A complex-valued function $u$ satisfies the equation \eqref{01} for a certain bounded {\it complex} potential $V$ if and only if
\begin{equation}
\label{02}
\left| \D u\right| \le C |u| .
\end{equation}
If the potential $V$ is real, the equation \eqref{01} is true for the real part $\re u$ and 
for the imaginary part $\im u$. 
It means that the condition that a function $u$ satisfies the equation \eqref{01} 
for a certain bounded {\it real} potential $V$ is equivalent to the estimate \eqref{02} 
with a real-valued function $u.$

For $n=1$ it is easy to see that any solution which decreases super\-exponen\-ti\-ally fast
\begin{equation*}
u(x) = O \big(e^{-N|x|}\big) \quad \forall \ N,
\end{equation*}
is $u \equiv 0$.
On the other hand, there exist nontrivial solutions which decrease exponen\-ti\-al\-ly. 
For example, the function $u(x) = \exp (-\sqrt{x^2+1})$ satisfies the estimate \eqref{02}. 
Therefore, for $n=1$ the answer is clear both in the real and in the complex cases.

The question in the multidimensional complex case was solved by Meshkov \cite{M}. 
He showed that any function $u$ saisfying \eqref{02} and 
\begin{equation*}
u(x) = O \big(e^{-N|x|^{4/3}}\big) \quad \forall \ N,
\end{equation*}
is $u \equiv 0$.
Also for $n=2$ he constructed a nontrivial function $u$ satisfying the estimate \eqref{02} 
and the estimate
$$
u(x) = O \big(e^{-c|x|^{4/3}}\big) 
$$ 
for a certain $c>0$.

The real multidimensional case remains an open problem. 
Recently in \cite{LMNN} it has been shown that for $n=2$,
 any real function  $u$ satisfying \eqref{02} and
$$u(x) = O \big( e^{-N |x|\sqrt{\ln|x|}}\big), \quad \text{for}\ |x| \to \infty, \quad \forall \ N$$
is $u \equiv 0$.

We consider an analogous question in a cylinder
$$
\Pi := \R \times \T^d, \qquad \T^d = \mathbb{R}^d / (2\pi\Z)^d .
$$
Our main result is the following

\begin{theorem}
\label{t04}
Let $d \ge 3$, $N>0$.
There exists a real-valued function $F \in C^\infty \left(\R\times \T^d\right),$ 
$F \not\equiv 0,$ such that
\begin{equation}
\label{07}
|\D F(x)| \le C |F(x)|,
\end{equation}
\begin{equation}
\label{08}
|F(x)| \le C e^{-N|w|^{4/3}}
\end{equation}
for all $x = (w,y) \in \R \times \T^d$.
\end{theorem}

\begin{rem}
If one constructs such function $F(w, y_1, y_2, y_3)$ for $d=3$,
then the same function works also as an example for $d>3$.

Assume that the function $F_1 (w,y)$ satisfies \eqref{07}, \eqref{08}
with $N=1$. 
Consider the function $\tilde F (w,y) = F_1 (nw,ny)$ with some natural $n$.
Then $\tilde F$ is also $(2\pi)$-periodic in each variable $y_j$,
$$
\left|\D \tilde F(w,y)\right| = n^2 \left|\D F_1 (nw, ny)\right| 
\le n^2 C \left|F_1 (w,y)\right|,
$$
$$
\left|\tilde F(w,y)\right| \le C e^{-n^{4/3} |w|^{4/3}}.
$$
So, if there is a function for $N=1$ then there is a such function for any $N$.

Thus, Theorem \ref{t04} follows from its particular case $d=3$ and $N=1$.
In the sequel we construct an example just for $d=3$ and $N=1$.
\end{rem}

Bourgain and Kenig proved the following estimate.

\begin{theorem}[\cite{BK}, Lemma 3.10]
\label{BK}
Assume that a function $u$ satisfies \eqref{02} in $\R^n$, $u \not\equiv 0$.
Then 
$$
\max_{|x-x_0|\le 1} |u(x)| \ge c' \exp \left(- c' |x_0|^{4/3} \ln |x_0|\right).
$$
\end{theorem}

They also asked \cite[end of \S 3]{BK} what can be said about the quantity
\begin{equation}
\label{L}
{\cal L} := \limsup_{|x_0|\to \infty} \min_{|x-x_0|\le 1} \frac{\ln \left|\ln |u(x)|\right|}{\ln |x|}
\end{equation}
for non-trivial solutions to \eqref{02}.
Theorem \ref{BK} shows that ${\cal L} \le 4/3$ always.
Our example shows that for $n\ge 4$ this estimate is sharp 
even for real-valued solutions. 
Indeed, extend the function $F$ from Theorem \ref{t04}
periodically in $\R^{d+1}$.
It satisfies \eqref{02} in $\R^{d+1}$.
Take $x = (w,0)$.
Then
$$
|F(x)| \le C e^{-|w|^{4/3}} = C e^{-|x|^{4/3}},
\qquad 
\ln \left|\ln |F(x)|\right| \ge \frac43 \ln |x| - 1
$$
for sufficiently large $|x|$, and therefore ${\cal L} = 4/3$.

\subsection{Relation with an abstract question}
In 2020 D.~Elton proved  

\begin{theorem}[\cite{Elton}]
\label{t01}
Let ${\mathcal A}$ be a nonnegative self-ajoint operator in the Hilbert space ${\mathcal H},$
${\mathcal A} = {\mathcal A}^* \ge 0$.
Suppose that the function $\phi \colon (0,\infty) \to {\mathcal H},$
\begin{equation}
\label{03}
\phi \in L_2\big((0,\infty); \dom {\mathcal A}\big) \cap 
W_2^1 \big((0,\infty); \dom \sqrt{\mathcal A}\big) \cap W_2^2 \big((0, \infty); {\mathcal H}\big)
\end{equation}
satisfies the inequality
\begin{equation}
\label{04}
\Big\|\frac{d^2 \phi (w)}{dw^2} - {\mathcal A} \phi (w)\Big\|_{\mathcal H} 
\le C \|\phi(w)\|_{\mathcal H} .
\end{equation}

\textup{1)} If we also have
$$
\int\limits_0^\infty \big\|\phi(w)\big\|_{\mathcal H}^2 \,e^{2N w^{4/3}} dw < \infty 
\quad \forall \ N > 0,
$$
then $\phi \equiv 0$.

\textup{2)} If the positive spectrum of ${\mathcal A}$ contains arbitrarily large gaps, and
$$
\int\limits_0^\infty \big\|\phi(w)\big\|_{\mathcal H}^2 \,e^{2N w} dw < \infty 
\quad \forall \ N > 0,
$$
then $\phi \equiv 0$.
\end{theorem}

This result is sharp.
Assume that all gaps in the spectrum of a self-adjoint operator ${\cal A}$
are bounded by some constant.
Then there exists a non-trivial function
$\phi \colon (0,\infty) \to {\mathcal H}$
satisfying \eqref{03}, \eqref{04} and the estimate
$\|\phi(w)\|_{\cal H} \le A e^{-Bw^{4/3}}$, see \cite[Theorem 0.2]{FKr}.

The equation \eqref{01} in the half-cylinder $[0, \infty) \times \T^d$ satisfies 
the condition of the Theorem \ref{t01},
if we take as ${\mathcal A}$ the Laplace operator over the variables $y$, 
${\mathcal A} = - \D_{\T^d}$ in the Hilbert space ${\mathcal H} = L_2 (\T^d)$.
The spectrum of $- \D_{\T^d}$ is the set of sums of $d$ squares,
$$
\si ({\mathcal A}) = \big\{ n_1^2 + \dots + n_d^2; \quad n_1, \dots, n_d \in \Z\big\} .
$$
Its structure is well studied.
For $d=1$ and $d=2$ there exist arbitrary long segments in the set of positive integers
which contain no numbers from this set. 
If $d=3,$ the set does not contain the numbers of the form $4^a (8b+7)$ only,
and for $d\ge 4$ it is the whole set of nonnegative integers
$\si ({\mathcal A}) = \N_0$.
Therefore, Theorem \ref{t01} implies the following 

\begin{theorem}
\label{t03}
Suppose that a complex-valued function $u \in W_2^2 \left((0,\infty) \times \T^d\right)$ satisfies the estimate
$$
|\D u(w,y)| \le C |u(w,y)| \quad \text{for almost all} \quad  (w,y) \in (0,\infty) \times \T^d .
$$

\textup{1)} If we also have
\begin{equation}
\label{05}
u (w,y) =  O \big(e^{-Nw^{4/3}}\big) \quad \forall \ N,
\end{equation}
then $u \equiv 0$.

\textup{2)} For $d=1$ and $d=2$ if
\begin{equation}
\label{06}
u (w,y) =  O \left(e^{-Nw}\right) \quad \forall \ N,
\end{equation}
then $u \equiv 0$.
\end{theorem}

Of course, this Theorem is sharp in the cases $d=1$ and $d=2$: 
as an example one may consider the function $u(w,y) = e^{-w}$.
In our previous work \cite{FKr} we showed that this Theorem is 
sharp for $d \ge 3$ in the complex case.
Now, Theorem \ref{t04} shows that in the real case it is also sharp.

In \S 1 we construct an explicit function $f$ in the cylinder which
satisfies bounds $f(x) = O \left(e^{-|w|^{4/3}}\right)$ and 
$\D f(x) = O \left(e^{-|w|^{4/3}}\right)$.
In \S 2 we describe the domain $\O$ where the values of the function $f$ are very small,
$|f(x)| \le h e^{-|w|^{4/3}}$ with a fixed number $h \ll 1$.
Section 3 contains some auxiliary facts.
In \S 4 and \S 5 we study the Dirichlet problem for the Poisson equation in the domain $\O$.
In \S 6 we construct first the function $f^*$ that coincides with the function $f$
outside $\O$, and is its harmonic extension in $\O$.
Finally, we define the function $F$ as the mollification of the function $f^*$,
and prove Theorem \ref{t04}.

We use the notations 
$\n f = \left(\frac{\dd f}{\dd w}, \frac{\dd f}{\dd y_1},
\frac{\dd f}{\dd y_2}, \frac{\dd f}{\dd y_3}\right)^t$,
$\n_y f = \left(\frac{\dd f}{\dd y_1},
\frac{\dd f}{\dd y_2}, \frac{\dd f}{\dd y_3}\right)^t$.
By $\mes_k A$ we denote the $k$-dimensional Lebesgue measure of  the set $A$.

The authors thank K\' evin Le Balc'h for useful discussions.

%%%%%%%%%%%%%%%%%%%%%%%%%%%%%%%%%%%%%%%
\section{Functions $f_0$ and $f$}
\label{s1}
\subsection{Sequence of integer vectors}

\begin{lemma}
\label{l11}
There is a sequence $\{a_k\}_{k=1}^\infty$, $a_k \in \Z^3$, such that
$$
a_k^2 = 4k+1, \qquad 
\left|a_k \cdot a_{k+1}\right| \le \frac1{\sqrt 3} \left|a_k\right| \left|a_{k+1}\right|.
$$
\end{lemma}

\begin{proof}
We construct this sequence step by step.
Assume that the vector $a_k$ is already chosen.
Without loss of generality we can assume that
$$
a_{k,1} \ge a_{k,2} \ge a_{k,3} \ge 0.
$$
Then
$$
a_k^2 \ge l a_{k,l}^2 \quad \Rightarrow \quad a_{k,l} \le \frac{|a_k|}{\sqrt l}, \quad l = 1, 2,  3.
$$
It is well known that any natural number of the form $(4m+1)$ 
can be represented as a sum of three squares.
Therefore, one can find a vector $b \in \Z^3$ such that $|b| = \sqrt{4k+5}$.
Moreover we can assume
$$
0 \le b_1 \le b_2 \le b_3,
$$
and
$$
b_l \le \frac{|b|}{\sqrt{4-l}}, \quad l = 1, 2,  3.
$$
Thus,
$$
a_{k,l} b_l \le \frac{|a_k| |b|}{\sqrt{l(4-l)}} \le \frac{|a_k| |b|}{\sqrt 3}, \quad l = 1, 2,  3.
$$
Let us consider four sums
\begin{eqnarray*}
a_{k,1} b_1 + a_{k,2} b_2 + a_{k,3} b_3 > 0,\\
a_{k,1} b_1 + a_{k,2} b_2 - a_{k,3} b_3 , \\
a_{k,1} b_1 - a_{k,2} b_2 - a_{k,3} b_3, \\
- a_{k,1} b_1 - a_{k,2} b_2 - a_{k,3} b_3 < 0.
\end{eqnarray*}
There are two neighbour lines where the sum in the first line is non-negative,
and the sum in the second line is non-positive.
They differ at most by $2 |a_k| |b|/\sqrt 3$.
Therefore, one of them is at most $|a_k| |b|/\sqrt 3$ in absolute value.
Now, we can take as $a_{k+1}$ the vector $b$ with suitable signs of its coordinates.
\end{proof}

\begin{cor}
\label{c12}
We have $(\sin \al_k)^2 \ge 2/3$ where $\al_k$ is the angle between $a_k$ and $a_{k+1}$.
\end{cor}

\subsection{Function $f_0$}
We need some auxiliary numbers and functions.
Introduce numbers
$$
w_k = \frac{27}8 \, k^{3/2}, \quad k \in \N.
$$
We have
$$
(k+1)^{3/2} - k^{3/2} = \frac32 \int_k^{k+1} \sqrt t\, dt,
$$
so
$$
\frac{81}{16} \, \sqrt k \le w_{k+1} - w_k \le \frac{81}{16} \, \sqrt {k+1}.
$$

Fix a function $\ze \in C^\infty (\R)$ possessing the following properties:
$$
\ze(w) = 0 \quad \text{if} \quad w \le 0, \qquad
\ze(w) = 1 \quad \text{if} \quad w \ge 1,
$$
\begin{equation}
\label{11}
0 \le \ze(w) \le 1 \quad \forall \ w, \qquad
\ze'(w) > 0 \quad \text{if} \quad 0 < w < 1,
\end{equation}
$$
\ze'(w) \le 2 \quad \forall \ w, \qquad |\ze''(w)| \le 5 \quad \forall \ w.
$$

Fix a number 
$$
m= 40000.
$$
For $w\ge 0$ we define the function
\begin{equation}
\label{13}
f_0(x) = \begin{cases}
\cos \left(a_{m+1} y\right), & 0 \le w \le w_{m+1}, \\
\left(1 - \ze\left(\frac{w-w_k}{w_{k+1}-w_k}\right)\right) \cos (a_k y)
+ \ze\left(\frac{w-w_k}{w_{k+1}-w_k}\right) \cos (a_{k+1} y),
& w_k \le w \le w_{k+1}, k \ge m+1,
\end{cases}
\end{equation}
where $\{a_k\}$ is the sequence of vectors constructed in Lemma \ref{l11}.
We extend $f_0$ as even function for $w \le 0$,
$$
f_0 ((w,y)) = f_0((-w,y)).
$$

\begin{lemma}
\label{l12}
The function $f_0$ possesses the following properties:
$$
f_0 \in C^\infty (\Pi), \qquad |f_0(x)| \le 1 \quad \forall \ x,
$$
$$
\frac{\dd f_0(x)}{\dd w} = 0, \quad
\left|\n_y f_0(x)\right| \le |a_{m+1}|, \qquad \text{if} \quad |w| \le w_{m+1},
$$
$$
|\n f_0(x)| \le \sqrt{5k}, \qquad \text{if} \quad w_k \le w \le w_{k+1}, \ k \ge m+1.
$$
\end{lemma}

\begin{proof}
First four claims are evident.
Prove the last one.

For $w_k \le w \le w_{k+1}$ we have
$$
\frac{\dd f_0(x)}{\dd w}
= \frac1{w_{k+1}-w_k}\, \ze' \left(\frac{w-w_k}{w_{k+1}-w_k}\right)
\left(\cos (a_{k+1} y) - \cos (a_k y)\right),
$$
therefore,
$$
\left|\frac{\dd f_0(x)}{\dd w}\right| \le \frac4{w_{k+1}-w_k} \le \frac{64}{81\sqrt k}.
$$
Clearly, 
$$
\left|\n_y f_0(x)\right| \le |a_{k+1}| = \sqrt{4k+5},
$$
so,
$$
|\n f_0(x)|^2 \le 4k+5+\frac1k \qquad 
\Rightarrow \qquad |\n f_0(x)| \le \sqrt{5k}
$$
if $w_k \le w \le w_{k+1}$, $k \ge m+1$.
\end{proof}

\subsection{Function $f$}
For $w\ge 0$ define the function
$$
\al(w) = \begin{cases}
w_m, & \text{if} \ 0 \le w \le w_m, \\
\left(1 - \ze\left(\frac{w-w_m}{w_{m+1}-w_m}\right)\right) w_m
+ \ze\left(\frac{w-w_m}{w_{m+1}-w_m}\right) w, & \text{if} \ w_m < w < w_{m+1}, \\
w, & \text{if} \ w \ge w_{m+1},
\end{cases}
$$
and extend it as an even function for $w\le 0$,
$$
\al(w) = \al(-w).
$$

\begin{lemma}
\label{l135}
The function $\al$ possesses the following properties:
$$
\al \in C^\infty (\R), \qquad \al(w) \ge w_m \quad \forall \ w, \qquad
\al(w) \le w \quad \text{if} \ w \ge w_m,
$$
\begin{equation}
\label{125}
|\al'(w)| \le 3, \quad \left|\al''(w)\right| \le \frac3{\al(w)^{1/3}}
\quad \forall \ w.
\end{equation}
\end{lemma}

\begin{proof}
First three claims are evident.
Inequalities \eqref{125} are also evident for $|w|\le w_m$ and for $|w| \ge w_{m+1}$.

For $w_m \le w \le w_{m+1}$ we have
\begin{equation}
\label{al1}
\al'(w) = 
\frac{w-w_m}{w_{m+1}-w_m}\, \ze' \left(\frac{w-w_m}{w_{m+1}-w_m}\right)
+ \ze\left(\frac{w-w_m}{w_{m+1}-w_m}\right) ,
\end{equation}
so $0 \le \al'(w) \le 3$.
Moreover,
$$
\al''(w) = 
\frac{w-w_m}{\left(w_{m+1}-w_m\right)^2}\, \ze'' \left(\frac{w-w_m}{w_{m+1}-w_m}\right)
+ \frac2{w_{m+1}-w_m} \, \ze' \left(\frac{w-w_m}{w_{m+1}-w_m}\right),
$$
and by virtue of \eqref{11}
$$
\left|\al''(w)\right| 
\le \frac9{w_{m+1}-w_m} \le \frac{16}{9 \sqrt m}
= \frac8{3 w_m^{1/3}} \le \frac3{w_{m+1}^{1/3}} \le \frac3{\al(w)^{1/3}} .
\qquad \qedhere
$$
\end{proof}

\begin{cor}
\label{c145}
We have
$$
\left|\al(w)^{4/3} - \al(\tilde w)^{4/3}\right| 
\le 4 |w-\tilde w| \max\left(\al(w), \al(\tilde w)\right)^{1/3}.
$$
\end{cor}

Put
$$
f(x) = f_0(x) \, e^{-\al(w)^{4/3}}.
$$

\begin{lemma}
\label{l13}
The function $f$ possesses the following properties:
$$
f((w,y)) = f((-w,y)), \qquad f \in C^\infty (\Pi), \qquad |f(x)| \le e^{-\al(w)^{4/3}} \ \forall \ x \in \Pi,
$$
$$
|\D f(x)| \le 40 m |f(x)| \quad \text{if} \ |w| \le w_{m+1}, \qquad
|\D f(x)| \le 10 e^{-|w|^{4/3}}, \quad \text{if} \ |w| \ge w_{m+1}.
$$
\end{lemma}

\begin{proof}
First three claims are evident.
If $|w| \le w_m$ then
$$
f(x) = \cos(a_{m+1}y) \,e^{-w_m^{4/3}},
$$
and
$$
\D f(x) = -(4m+5) f(x) \qquad \Rightarrow \qquad 
|\D f(x)| \le 40 m |f(x)|.
$$

Now, let $w_k \le w \le w_{k+1}$, $k \ge m$.
We have
$$
\left(e^{-\al(w)^{4/3}}\right)' 
= - \frac43 \al(w)^{1/3} \al'(w) e^{-\al(w)^{4/3}},
$$
and
\begin{equation}
\label{al3}
\left(e^{-\al(w)^{4/3}}\right)''
= \left(\frac{16}9 \al^{2/3} (\al')^2 - \frac49 \al^{-2/3} (\al')^2 
- \frac43 \al^{1/3}\al''\right) e^{-\al(w)^{4/3}}.
\end{equation}
If $k=m$ and $w_m \le w \le w_{m+1}$ then
$$
f(x) = \cos(a_{m+1}y) \,e^{-\al(w)^{4/3}},
$$
\begin{eqnarray*}
\D f(x) = - a_{m+1}^2 \cos (a_{m+1}y) e^{-\al(w)^{4/3}} 
+ \cos (a_{m+1}y) \left(e^{-\al(w)^{4/3}}\right)''\\
= \left(\frac{16}9 \al^{2/3} (\al')^2 - (4m+5) - \frac49 \al^{-2/3} (\al')^2 
- \frac43 \al^{1/3}\al''\right) f(x).
\end{eqnarray*}
Here 
$$
w_m \le \al(w) \le w_{m+1}, \quad 0 \le \al'(w) \le 3, \quad |\al''(w)| \le 3 \al(w)^{-1/3},
$$
therefore,
$$
|\D f(x)| \le 16 w_{m+1}^{2/3} |f(x)| = 36 (m+1) |f(x)| \le 40 m |f(x)|.
$$

Finally, if $w_k \le w \le w_{k+1}$, $k \ge m+1$, then
the formula \eqref{al3} transforms into
$$
\left(e^{-w^{4/3}}\right)'' = 
\left(\frac{16}9\, w^{2/3} - \frac49\, w^{-2/3}\right) e^{-w^{4/3}},
$$
and
$$
\D \left(e^{-w^{4/3}} \cos (a_k y)\right)
= \left(\frac{16}9\, w^{2/3} - \frac49\, w^{-2/3} - a_k^2\right) e^{-w^{4/3}} \cos (a_k y).
$$
Moreover,
$$
\frac{16}9\,w^{2/3} \in [4k, 4k+4], \qquad a_k^2 = 4k+1,
$$
so
$$
\left|\D \left(e^{-w^{4/3}} \cos (a_k y)\right)\right| \le 3 e^{-w^{4/3}}.
$$
In the same manner,
$$
\left|\D \left(e^{-w^{4/3}} \cos (a_{k+1} y)\right)\right| \le 6 e^{-w^{4/3}}.
$$
By \eqref{13}
\begin{eqnarray*}
\D f(x) = \left(1 - \ze\left(\frac{w-w_k}{w_{k+1}-w_k}\right)\right) 
\D \left(e^{-w^{4/3}}\cos (a_k y)\right)\\
+ \ze\left(\frac{w-w_k}{w_{k+1}-w_k}\right) \D \left(e^{-w^{4/3}}\cos (a_{k+1} y)\right) \\
+ \frac2{w_{k+1}-w_k} \, \ze'\left(\frac{w-w_k}{w_{k+1}-w_k}\right)
\left(e^{-w^{4/3}}\right)' \left(\cos(a_{k+1} y) - \cos(a_k y)\right) \\
+ \frac1{(w_{k+1}-w_k)^2} \, \ze''\left(\frac{w-w_k}{w_{k+1}-w_k}\right)
e^{-w^{4/3}} \left(\cos(a_{k+1} y) - \cos(a_k y)\right) .
\end{eqnarray*}
Therefore,
\begin{eqnarray*}
\left|\D f(x)\right| \le 
6 e^{-w^{4/3}} + \frac8{w_{k+1}-w_k} \cdot \frac43 w^{1/3} e^{-w^{4/3}}
+ \frac{10}{(w_{k+1}-w_k)^2} e^{-w^{4/3}} \\
\le \left(6 + \frac{256 \sqrt{k+1}}{81 \sqrt k} + \frac{2560}{6561 k}\right) e^{-w^{4/3}}
\le 10 e^{-w^{4/3}}. \qquad \qedhere
\end{eqnarray*}
\end{proof}

\begin{cor}
\label{c155}
For all $x \in \Pi$
$$
|\D f(x)| \le 40 m e^{-\al(w)^{4/3}}.
$$
\end{cor}

%%%%%%%%%%%%%%%%%%%%%%%%%%%%%%%%%%%%%%%%%%
\section{Domain $\O$}

Denote
\begin{equation}
\label{h}
h = 10^{-8}
\end{equation}
and introduce the set
$$
\O = \left\{ x \in \Pi : |f_0(x)| < 2h\right\}.
$$

\subsection{The boundary of $\O$}
\begin{lemma}
\label{l21}
If $|f_0(x)| = 2h$ then $\n f_0(x) \neq 0$.
\end{lemma}

\begin{proof}
It is sufficient to consider $w \ge 0$ only.

For $0 \le w \le w_{m+1}$ we have
$$
f_0 (x) = \cos \left(a_{m+1} y\right), \qquad 
\n_y f_0(x) = - \sin \left(a_{m+1} y\right) a_{m+1}, 
$$
so if $\left|\cos \left(a_{m+1} y\right)\right| = 2h$ then $\n_y f_0(x) \neq 0$.

The same argument works for the case $x = (w_k, y)$.

Now, consider the case 
$w_k < w < w_{k+1}$, $k \ge m+1$.
Due to \eqref{13}
$$
\frac{\dd f_0(x)}{\dd w}
= \frac1{w_{k+1}-w_k}\, \ze' \left(\frac{w-w_k}{w_{k+1}-w_k}\right)
\left(\cos (a_{k+1} y) - \cos (a_k y)\right),
$$
$$
\n_y f_0(x) = - \left(1- \ze \left(\frac{w-w_k}{w_{k+1}-w_k}\right)\right)
\sin (a_k y) \,a_k
- \ze \left(\frac{w-w_k}{w_{k+1}-w_k}\right) \sin (a_{k+1} y) \, a_{k+1}.
$$
Now, assume that $\n f_0(x) = 0$.
The vectors $a_k$ and $a_{k+1}$ are linearly independent,
so by virtue of \eqref{11} we obtain
$$
\begin{cases}
\cos (a_{k+1} y) - \cos (a_k y) = 0, \\
\sin (a_k y) = 0, \\
\sin (a_{k+1} y) = 0.
\end{cases}
$$
This implies 
$$
f_0(x) = \cos (a_k y) = \pm 1.
$$
Thus, $|f_0(x)| \neq 2h$, a contradiction.
\end{proof}

\begin{cor}
\label{c22}
The boundary $\dd \O$ is $C^\infty$-smooth.
\end{cor}

\subsection{``Sparcity'' of $\O$}
\begin{lemma}
\label{l23}
For $0 \le \eta \le 1$ and $\si \in \R$ denote
$$
A_\si := \left\{\te \in [0,\pi] : \cos \te \in [\si - \eta, \si]\right\}.
$$
Then $\mes_1 A_\si \le 2 \sqrt\eta$.
\end{lemma}

\begin{proof}
Without loss of generality one can assume $\si \le 1$.
Moreover,
$$
A_{\eta-\si} = \left\{\pi - \te : \te \in A_\si\right\},
$$
therefore $\mes_1 A_{\eta-\si} = \mes_1 A_\si$.
So, one can assume that $\si \ge \eta/2$.
Then we have
$$
\mes_1 A_\si = \arccos (\si - \eta) - \arccos \si,
$$
$$
\frac{d(\mes_1 A_\si)}{d\si} 
= \frac1{\sqrt{1-\si^2}} - \frac1{\sqrt{1-(\si-\eta)^2}} \ge 0
\quad \text{for} \quad \frac\eta2 \le \si \le 1.
$$
Thus,
$$
\mes_1 A_\si \le \mes_1 A_1 = \arccos (1-\eta) \le 2 \sqrt \eta,
$$
because the function
$$
\ph (\eta) : = 2 \sqrt \eta - \arccos (1-\eta) 
$$
increases on $[0,1]$ and $\ph(0) = 0$.
\end{proof}

\begin{cor}
\label{c24}
Let $0 \le \eta \le 1$, $\si \in \R$.

a) We have
$$
\mes_1 \left\{\te \in [-\pi,\pi] : \cos \te \in [\si - \eta, \si]\right\} \le 4 \sqrt \eta.
$$

b) Let $\al \in \R$. 
Then
$$
\mes_1 \left\{\te \in [\al-\pi,\al+\pi] : \cos \te \in [\si - \eta, \si]\right\} \le 4 \sqrt \eta.
$$

c) Let $\mu > 0$, $\al \in \R$. 
Then
$$
\mes_1 \left\{\tau \in \left[-\frac\pi\mu,\frac\pi\mu\right] : 
\cos (\al + \mu\tau) \in [\si - \eta, \si]\right\} \le \frac{4 \sqrt \eta}\mu.
$$

d) Let $\mu \ge \pi$, $A > 0$.
Then
$$
\mes_1 \left\{\tau \in [-A,A] : 
\cos (\al + \mu\tau) \in [\si - \eta, \si]\right\} \le \frac{4 (A+1) \sqrt \eta}\pi.
$$
\end{cor}

\begin{proof}
a) Follows from the evenness of cosine.

b) Follows from the periodicity of cosine.

c) The change of variables $\te = \al + \mu \tau$ transforms the set
$$
\left\{\tau \in \left[-\frac\pi\mu,\frac\pi\mu\right] : 
\cos (\al + \mu\tau) \in [\si - \eta, \si]\right\} 
$$
into the set
$$
\left\{\te \in [\al-\pi,\al+\pi] : \cos \te \in [\si - \eta, \si]\right\}.
$$

d) After the same change of variables $\te = \al + \mu \tau$
we have to consider the set
$$
\left\{\te \in [\al-\mu A, \al+\mu A] : \cos \te \in [\si - \eta, \si]\right\}.
$$ 
We cover the interval $[\al-\mu A, \al+\mu A]$ by $\left(\left[\frac{\mu A}\pi\right]+1\right)$
intervals of length $2\pi$.
Then by the claim of the point b)
$$
\mes_1 \left\{\tau \in [-A,A] : 
\cos (\al + \mu\tau) \in [\si - \eta, \si]\right\} \le 
\left(\left[\frac{\mu A}\pi\right]+1\right) \frac{4 \sqrt \eta}\mu.
$$
By assumption $\mu \ge \pi$, so
$$
\left[\frac{\mu A}\pi\right]+1 %\le \frac{\mu A}\pi + 1 
\le \frac{\mu(A+1)}\pi.
$$
Therefore,
$$
\left(\left[\frac{\mu A}\pi\right]+1\right) \frac{4 \sqrt \eta}\mu
\le \frac{4 (A+1) \sqrt \eta}\pi. \qquad \qedhere
$$
\end{proof}

\begin{lemma}
\label{l25}
a) Let 
$$
y_* \in \T^3, \qquad 0 \le w \le w_{m+1} \qquad \text{and} \qquad
0 < R \le \frac\pi{\sqrt{4m+5}}.
$$
Then
$$
\mes_3 \left\{y \in \T^3 : |y-y_*| \le R, (w,y) \in \O\right\}
\le \frac{8 \sqrt{h}}{\sqrt{4m+5}} \, \pi R^2.
$$

b)  Let 
$$
y_* \in \T^3, \qquad w_k \le w \le w_{k+1}, \qquad k \ge m+1 \qquad \text{and} \qquad
0 < R \le \frac\pi{\sqrt{4k+5}}.
$$
Then
$$
\mes_3 \left\{y \in \T^3 : |y-y_*| \le R, (w,y) \in \O\right\}
\le \frac{8 \sqrt{3h}}{\sqrt{4k+1}} \, \pi R^2.
$$
\end{lemma}

\begin{proof}
a) We consider the set of $y \in \T^3$ such that
$|y-y_*| \le R$ and $\left|\cos\left(a_{m+1}y\right)\right| \le 2h$.
Clearly,
\begin{eqnarray*}
\mes_3 \left\{ y \in \T^3 : |y-y_*| \le R, \left|\cos\left(a_{m+1}y\right)\right| \le 2h\right\} \\
\le \pi R^2 \mes_1 \left\{\tau \in [-R,R] : 
\left|\cos\left(y_* a_{m+1} + |a_{m+1}| \tau\right)\right| \le 2 h\right\}
\le \frac{8 \sqrt{h}}{\sqrt{4m+5}} \, \pi R^2
\end{eqnarray*}
due to Corollary \ref{c24} c) with $\eta = 4h$, $\al = y_*a_{m+1}$ and $\mu = |a_{m+1}|$.

b) We have either
\begin{equation}
\label{23}
\ze \left(\frac{w-w_k}{w_{k+1}-w_k}\right) \ge \frac12,
\end{equation}
or
\begin{equation}
\label{24}
1-\ze \left(\frac{w-w_k}{w_{k+1}-w_k}\right) \ge \frac12.
\end{equation}
%or the both inequalities hold.
First, assume \eqref{24}.
Denote by $e_k$ the unit vector that belongs to the plane generated
by the vectors $a_k$ and $a_{k+1}$, and such that 
$a_{k+1} \cdot e_k = 0$ and $a_k \cdot e_k \ge 0$.
Then
\begin{equation}
\label{25}
a_k \cdot e_k = |a_k| |\sin \al_k| \ge \sqrt{\frac23} \, |a_k|
\end{equation}
by virtue of Corollary \ref{c12}.
Introduce the set 
$$
Q_R := \left\{y = y_* + z + \tau e_k: 
z \perp e_k, |z| \le R, |\tau| \le R\right\}.
$$
We have
$$
\left\{y \in \T^3 : |y-y_*| \le R\right\} \subset Q_R.
$$
If $y = y_* + z + \tau e_k$ and $(w,y) \in \O$ then
\begin{eqnarray*}
|f_0(w,y)| \\ 
= \left|\left(1 - \ze\left(\frac{w-w_k}{w_{k+1}-w_k}\right)\right) \cos (a_k (y_*+z+\tau e_k))
+ \ze\left(\frac{w-w_k}{w_{k+1}-w_k}\right) \cos (a_{k+1} (y_*+z))\right| \\
\le 2h.
\end{eqnarray*}
Taking into account \eqref{24} we obtain
\begin{equation}
\label{26}
\left|\cos \left(a_k (y_*+z) + \tau a_k\cdot e_k)\right)
+ \frac{\ze\left(\frac{w-w_k}{w_{k+1}-w_k}\right)}{1 - \ze\left(\frac{w-w_k}{w_{k+1}-w_k}\right)}
\,\cos (a_{k+1} (y_*+z))\right| \le 4h.
\end{equation}
Now we apply Corollary \ref{c24} c) with
$$
\eta = 8 h, \quad \al = a_k (y_*+z), \quad
\mu = a_k \cdot e_k \in \left[\sqrt{\frac23} \, |a_k|, |a_k|\right],
$$
where the last inclusion is due to \eqref{25}.
Taking into account the assumption
$$
R \le \frac\pi{\sqrt{4k+5}} < \frac\pi\mu, 
$$
we get that for any fixed vector $z$ 
$$
\mes_1 \left\{\tau \in [-R,R] : \text{the estimate \eqref{26} holds}\right\} 
\le \frac{4\sqrt \eta}\mu 
\le \frac{8 \sqrt{3h}}{\sqrt{4k+1}}.
$$
Therefore,
\begin{eqnarray}
\label{27}
\mes_3 \left\{y \in \T^3 : |y-y_*| \le R, (w,y) \in \O\right\} \\
\le \mes_3 \left\{y \in Q_R : (w,y) \in \O\right\} 
\le \pi R^2 \frac{8 \sqrt{3h}}{\sqrt{4k+1}}.
\nonumber
\end{eqnarray}
In the same way, if \eqref{23} holds true then the estimate \eqref{27} is also fulfilled.
\end{proof}

\begin{cor}
\label{c26}
Let $x_* = (w_*, y_*)$.

a) If $0 \le w_* \le w_{m+1}$ and $0 < R_* \le \frac\pi{\sqrt{4m+9}}$ then
$$
\mes_4 \left(B_R (x_*) \cap \O\right)  \le \frac{16 \sqrt{3h}}{\sqrt{4m+5}} \,\pi R^3.
$$

b) If $w_k \le w_* \le w_{k+1}$, $k \ge m+1$ and $0 < R \le \frac\pi{\sqrt{4k+9}}$ then
$$
\mes_4 \left(B_R (x_*) \cap \O\right)  \le \frac{16 \sqrt{3h}}{\sqrt{4k-3}} \,\pi R^3.
$$
\end{cor}

\begin{proof}
a) If $|w-w_*| < R$ then $|w| \le w_{m+2}$.
For such $w$ we have
$$
\mes_3 \left\{y \in \T^3 : |y-y_*| \le R, (w,y) \in \O\right\}
\le \frac{8 \sqrt{3h}}{\sqrt{4m+5}} \, \pi R^2.
$$
Therefore,
$$
\mes_4 \left(B_R (x_*) \cap \O\right)  \le \frac{16 \sqrt{3h}}{\sqrt{4m+5}} \,\pi R^3.
$$

b) If $|w-w_*| < R$ then $w_{k-1} \le w \le w_{k+2}$,
$$
\mes_3 \left\{y \in \T^3 : |y-y_*| \le R, (w,y) \in \O\right\}
\le \frac{8 \sqrt{3h}}{\sqrt{4k-3}} \, \pi R^2.
$$
Therefore,
$$
\mes_4 \left(B_R (x_*) \cap \O\right)  \le \frac{16 \sqrt{3h}}{\sqrt{4k-3}} \,\pi R^3.
\qquad\qedhere
$$
\end{proof}

\begin{cor}
\label{c261}
Assume that either the conditions of Corollary \ref{c26} a) are fulfilled and
$R \ge \frac{64\sqrt{3h}}{\pi\sqrt{4m+5}}$,
or the conditions of Corollary \ref{c26} b) are fulfilled and
$R \ge \frac{64\sqrt{3h}}{\pi\sqrt{4k-3}}$.
Then
$$
\mes_4 \left(B_R (x_*) \cap \O\right)  \le \frac12 \mes_4 B_R.
$$
\end{cor}

\begin{lemma}
\label{l28}
For any $w \in \R$
$$
\mes_3 \left\{y \in \T^3 : (w,y) \in \O\right\} \le \pi^3.
$$
\end{lemma}

\begin{proof}
If $|w| \le w_{m+1}$ then 
$$
\left\{y \in \T^3 : (w,y) \in \O\right\} 
= \left\{y \in \T^3 : \left|\cos \left(a_{m+1} y\right)\right| < 2h\right\}.
$$
We have
$$
\operatorname{diam} [-\pi,\pi]^3 = 2 \pi \sqrt 3,
$$
therefore,
$$
[-\pi,\pi]^3 \subset \left\{y = z + \tau \frac{a_{m+1}}{|a_{m+1}|} : 
z \perp a_{m+1}, |z| \le \pi \sqrt 3, |\tau| \le \pi \sqrt 3\right\}.
$$
By virtue of Corollary \ref{c24} d) with 
$$
A = \pi\sqrt 3, \quad \al=0, \quad \mu = |a_{m+1}|,  \quad \eta = 4h
$$
we have
$$
\mes_1 \left\{\tau \in [-\pi \sqrt 3, \pi \sqrt 3] :
\left|\cos\left(\tau|a_{m+1}|\right)\right| \le 2h\right\} 
\le \frac{8(\pi\sqrt3+1)\sqrt{h}}\pi.
$$
Therefore,
$$
\mes_3 \left\{y \in \T^3 : \left|\cos \left(a_{m+1} y\right)\right| < 2h\right\}
\le 24 \pi^2 (\pi\sqrt3+1)\sqrt{h} < \pi^3
$$
due to \eqref{h}.

Now, let $w_k \le w \le w_{k+1}$, $k \ge m+1$.
First, assume \eqref{24} to be fulfilled.
Fix the vector $e_k$ as in the proof of Lemma \ref{l25}.
Then
$$
[-\pi,\pi]^3 \subset \left\{y = z + \tau e_k : 
z \perp e_k, |z| \le \pi \sqrt 3, |\tau| \le \pi \sqrt 3\right\}.
$$
If $(w,y) \in \O$ then $|f_0(x)|\le 2h$ and the inequality \eqref{26}
holds true with $y_*=0$.
Now we apply Corollary \ref{c24} d) with
$$
A = \pi \sqrt 3, \quad
\eta = 8 h, \quad \al = a_k \cdot z, \quad
\mu = a_k \cdot e_k \in \left[\sqrt{\frac23} \, |a_k|, |a_k|\right].
$$
We obtain that for any fixed vector $z$ 
$$
\mes_1 \left\{\tau \in [-\pi \sqrt 3, \pi \sqrt 3] : \text{the estimate \eqref{26} holds}\right\} 
\le \frac{4\left(\pi \sqrt 3 + 1\right)\sqrt{8 h}}\pi 
< 32 \sqrt{h}.
$$
The same is true if the inequality \eqref{23} is fulfilled.
Therefore,
\begin{eqnarray*}
\mes_3 \left\{y \in \T^3 : (w,y) \in \O\right\} \\
\le \mes_3 \left\{y = z + \tau e_k : 
z \perp e_k, |z| \le \pi \sqrt 3, |\tau| \le \pi \sqrt 3, (w,y) \in \O\right\} \\
\le \pi (\pi \sqrt 3)^2 \cdot 32 \sqrt h < \pi^3
\end{eqnarray*}
by definition \eqref{h} of the number $h$.
\end{proof}

%%%%%%%%%%%%%%%%%%%%%%%%%%%%%%%%%%%%%%%%%%
\section{Auxiliary facts}

\begin{lemma}
\label{l31}
If ${\cal O}$ is an open subset of $\T^3$ and $\mes_3 {\cal O} \le \pi^3$ then
$$
\|u\|_{L_2({\cal O})}^2 \le \frac{15}7 \|\n u\|_{L_2({\cal O})}^2
\qquad \forall \ u \in \mathring W_2^1 ({\cal O}).
$$
\end{lemma}

\begin{proof}
It is sufficient to prove that
$$
\|u\|_{L_2(\T^3)}^2 \le \frac{15}7 \|\n u\|_{L_2(\T^3)}^2
$$
for all $u \in W_2^1 (\T^3)$ such that
$$
\mes_3 \left\{y \in \T^3 : u(y) = 0\right\} \ge 7 \pi^3.
$$
Represent a function $u$ as a sum
$$
u(y) = u_0 + v(y),
$$
where $u_0$ is a constant, and $\int_{\T^3} v(y)\,dy = 0$.
Then
$$
\|u\|_{L_2(\T^3)}^2 = 8 \pi^3 |u_0|^2 + \|v\|_{L_2(\T^3)}^2,
\qquad
\|v\|_{L_2(\T^3)}^2 \ge 7 \pi^3 |u_0|^2,
$$
and
\begin{equation}
\label{30}
\|u\|_{L_2(\T^3)}^2 \le \frac{15}7 \|v\|_{L_2(\T^3)}^2.
\end{equation}
Expanding the function $v$ into the Fourier series we see that
$$
\|v\|_{L_2(\T^3)}^2 \le \|\n v\|_{L_2(\T^3)}^2 = \|\n u\|_{L_2(\T^3)}^2.
$$
The reference to \eqref{30} completes the proof.
\end{proof}

\begin{lemma}[Friedrichs' inequality for $\O$]
\label{l32}
We have
$$
\int_\O |u(x)|^2 dx \le \frac{15}7 \int_\O |\n_y u(x)|^2 dx 
\qquad \forall \ u \in \mathring W_2^1 (\O).
$$
\end{lemma}

\begin{proof}
Follows from Lemma \ref{l28} and Lemma \ref{l31}.
\end{proof}

%\begin{theorem}[Schauder Theorem, \cite{LU}, Chapter III, Theorem 1.1]
%???
%\end{theorem}

\begin{theorem}[\cite{EG}, \S 5.7.1, Lemma 1]
\label{t33}
Let ${\cal O} \subset \R^n$ be a domain with smooth boundary,
$\eta \in C_0^1 (\R^n, \R^n)$, $x_0 \in {\cal O}$.
Then for almost all $r \in (0, \infty)$ 
$$
\int_{{\cal O} \cap B_r (x_0)} \div \eta\, dx
= \int_{\dd {\cal O} \cap B_r (x_0)} \eta \cdot \nu\, dS(x) 
+ \int_{{\cal O} \cap \dd B_r (x_0)} \eta \cdot \nu\, dS(x) .
$$
Here $\nu(x)$ is the vector of outward unit normal at the point $x$
to $\dd {\cal O}$ in the first integral and to $\dd B_r (x_0)$ 
in the second integral in the right hand side.
\end{theorem}

\begin{rem}
The same is true for functions $\eta \in C^1 \left(\overline{\cal O}, \R^n\right)$
because we can extend such functions to the entire space $\R^n$ 
via the smooth boundary $\dd {\cal O}$.
\end{rem}

\begin{cor}[Green's identities]
\label{c35}
Let
$$
{\cal O} \subset \R^n, \quad \dd {\cal O} \in C^\infty, \quad
x_0 \in {\cal O}, \quad u, v \in C^2 (\overline{\cal O}).
$$
For almost all $r \in (0, \infty)$ we have
\begin{equation}
\label{gr1}
\int_{{\cal O} \cap B_r (x_0)} \D v(x) \, dx 
= \int_{\dd {\cal O} \cap B_r (x_0)} \frac{\dd v}{\dd\nu} \, dS(x) 
+ \int_{{\cal O} \cap \dd B_r (x_0)} \frac{\dd v}{\dd\nu} \, dS(x) ,
\end{equation}
and
\begin{eqnarray}
\label{gr2}
\int_{{\cal O} \cap B_r (x_0)} \left(u \,\D v - v \,\D u\right) \, dx \\
= \int_{\dd {\cal O} \cap B_r (x_0)} \left(u \,\frac{\dd v}{\dd\nu} 
- v \,\frac{\dd u}{\dd\nu}\right) dS (x) 
+ \int_{{\cal O} \cap \dd B_r (x_0)} \left(u \,\frac{\dd v}{\dd\nu} 
- v \,\frac{\dd u}{\dd\nu}\right) dS (x).
\nonumber
\end{eqnarray}
\end{cor} 

Recall that in four-dimensional case the Green function for the Laplace operator is
$$
\frac1{4 \pi^2 |x-x_0|^2};
$$
on the boundary $\dd B_r(x_0)$ we have
$$
\frac{\dd}{\dd\nu}\, \frac1{4 \pi^2 |x-x_0|^2}
= \frac{-1}{2 \pi^2 |x-x_0|^3}.
$$

\begin{cor}
\label{c36}
Let
$$
n = 4, \quad {\cal O} \subset \R^4, \quad \dd {\cal O} \in C^\infty, \quad
x_0 \in {\cal O}, \quad v \in C^2 (\overline{\cal O}).
$$
Then for almost all $r \in (0, \infty)$ 
\begin{eqnarray*}
v(x_0) = - \int_{{\cal O} \cap B_r (x_0)} \frac{\D v(x) \, dx}{4 \pi^2 |x-x_0|^2} \\
+ \int_{\dd {\cal O} \cap B_r (x_0)} \left(\frac{\dd v(x)}{\dd\nu} \, \frac1{4 \pi^2 |x-x_0|^2}
- v(x) \, \frac{\dd}{\dd\nu(x)} \, \frac1{4 \pi^2 |x-x_0|^2}\right) dS (x) \\
+ \frac1{4\pi^2 r^2} \int_{{\cal O} \cap \dd B_r (x_0)} \frac{\dd v(x)}{\dd\nu} \, dS(x)
+ \frac1{2\pi^2 r^3} \int_{{\cal O} \cap \dd B_r (x_0)} v(x) \, dS(x) .
\end{eqnarray*}
\end{cor}

%%%%%%%%%%%%%%%%%%%%%%%%%%%%%%%%%%%%%%%%%%
\section{Dirichlet problem in $\O$}
\subsection{Functional $J$}

Let $\ph \in L_2(\O)$.
We consider the functional
\begin{equation}
\label{41}
J[v] = \int_\O \left(\frac{|\n v|^2}2 - \ph \, v\right) dx, 
\qquad v \in \mathring W_2^1 (\O).
\end{equation}
By Lemma \ref{l32} we have
$$
\left|\int_\O \ph\, v\, dx\right| \le \|v\|_{L_2(\O)} \|\ph\|_{L_2(\O)}
\le \sqrt{\frac{15}7} \|\n v\|_{L_2(\O)} \|\ph\|_{L_2(\O)},
$$
and therefore, the functional $J$ is semi-bounded from below,
$$
J[v] \ge \frac14 \|\n v\|_{L_2(\O)}^2 - \frac{15}7 \|\ph\|_{L_2(\O)}^2.
$$
In a standard way we have the following

\begin{theorem}
\label{t41}
Let the functional $J$ be defined by \eqref{41} with $\ph \in L_2(\O)$.

a) Functional $J$ attains its minimum on $\mathring W_2^1 (\O)$.

b) The Dirichlet problem
\begin{equation}
\label{415}
\begin{cases}
- \D u (x) = \ph (x) \quad \text{in} \ \O, \\
\left.u\right|_{\dd\O} = 0,
\end{cases}
\end{equation}
has a unique solution $u \in \mathring W_2^1 (\O)$.

c) This function $u$ is a minimizer of the functional $J$.
The minimizer is unique.
\end{theorem}

\begin{rem} 
Under the conditions of Theorem \ref{t41} 
\begin{equation}
\label{42}
\|u\|_{W_2^1(\O)} \le 3 \|\ph\|_{L_2(\O)}.
\end{equation}
Indeed, 
$$
\|\n u\|_{L_2(\O)}^2 = \int_\O u \, \ph\, dx 
\le \|u\|_{L_2(\O)} \|\ph\|_{L_2(\O)}.
$$
By virtue of Lemma \ref{l32} 
$\|u\|_{L_2(\O)} \le \sqrt{\frac{15}7} \|\n u\|_{L_2(\O)}$,
therefore,
$$
\|\n u\|_{L_2(\O)} \le \sqrt{\frac{15}7} \|\ph\|_{L_2(\O)}, 
$$
\begin{equation}
\label{435}
\|u\|_{L_2(\O)} \le \frac{15}7 \|\ph\|_{L_2(\O)},
\end{equation}
and so
$$
\|\n u\|_{L_2(\O)}^2 + \|u\|_{L_2(\O)}^2 \le 
\left(\frac{15}7 + \frac{225}{49}\right) \|\ph\|_{L_2(\O)}^2 \le 9 \|\ph\|_{L_2(\O)}^2.
$$
\end{rem}

\begin{lemma}
\label{l43}
If under the assumptions of Theorem \ref{t41} 
$\ph(x) \ge 0$ a.e. in $\O$, then $u(x) \ge 0$ a.e. in $\O$.
\end{lemma}

\begin{proof}
The function $u$ is the minimizer of $J$.
Consider the function
$$
u_+ (x) = \max (u(x), 0).
$$
It is well known that $u_+ \in \mathring W_2^1 (\O)$ and
$$
\n u_+ (x) = \begin{cases}
\n u(x), \quad &\text{if} \ u(x) > 0, \\
0, \quad &\text{if} \ u(x) \le 0.
\end{cases}
$$
Therefore,
$$
J[u_+] = \int_\O \left(\frac{|\n u_+|^2}2 - \ph \, u_+\right) dx \le J[u]
$$
due to the assumption $\ph \ge 0$.
Thus, the function $u_+$ is also a minimizer for the functional $J$.
Therefore, $u_+ = u$.
\end{proof}

\subsection{Dirichlet problem in $\O$}

\begin{lemma}
\label{l44}
Let $u$ be the solution to the problem \eqref{415} with
$\ph \in L_2(\O) \cap C^\ga (\overline\O)$, $\ga > 0$.
Then $u \in C^{2+\ga} (\overline\O)$.
\end{lemma}

%\begin{proof}
%???
%\end{proof}

\begin{lemma}
\label{l45}
Let $\ph \in L_2(\O) \cap C^\ga (\overline\O)$, $\ga>0$, and $\ph(x) \ge 0$ for all $x \in \O$.
Let $u \in \mathring W_2^1 (\O)$ be the solution to the Dirichlet problem \eqref{415}.
Let $x_0 = (w_0, y_0)$.
Then for almost all $r\in (0,1]$ we have
\begin{equation}
\label{43}
u (x_0) \le \int_{\O \cap B_r(x_0)} \frac{\ph (x) \, dx}{4 \pi^2 |x-x_0|^2} 
+ \frac1{2 \pi^2 r^3} \int_{\O \cap \dd B_r(x_0)} u(x) \, dS(x).
\end{equation}
\end{lemma}

\begin{proof}
By Lemma \ref{l44} $u \in C^2 (\overline\O)$.
Now we use the identity \eqref{gr1} and Corollary \ref{c36}.
For almost all $r>0$ 
\begin{equation}
\label{44}
0 = \int_{\O \cap B_r(x_0)} \ph(x) \, dx 
+ \int_{\dd\O \cap B_r(x_0)} \frac{\dd u}{\dd\nu} \, dS(x) 
+ \int_{\O \cap \dd B_r(x_0)} \frac{\dd u}{\dd\nu} \, dS(x) ,
\end{equation}
and
\begin{eqnarray}
\nonumber
u(x_0) = \int_{\O \cap B_r (x_0)} \frac{\ph (x) \, dx}{4 \pi^2 |x-x_0|^2} 
+ \int_{\dd \O \cap B_r (x_0)} \frac{\dd u(x)}{\dd\nu} \, \frac1{4 \pi^2 |x-x_0|^2} dS (x) \\
+ \frac1{4\pi^2 r^2} \int_{\O \cap \dd B_r (x_0)} \frac{\dd u(x)}{\dd\nu} \, dS(x)
+ \frac1{2\pi^2 r^3} \int_{\O \cap \dd B_r (x_0)} u(x) \, dS(x) ,
\label{45}
\end{eqnarray}
where we used the boundary condition 
$\left.u\right|_{\dd\O} = 0$.
Multiplying \eqref{44} by $\frac1{4\pi^2r^2}$ and substracting from \eqref{45} we obtain
\begin{eqnarray}
\label{46}
u(x_0) = \int_{\O \cap B_r (x_0)} \frac{\ph(x)}{4\pi^2}
\left(\frac1{|x-x_0|^2} - \frac1{r^2}\right) dx \\
+ \int_{\dd \O \cap B_r (x_0)} \frac{\dd u(x)}{\dd\nu} \frac1{4\pi^2}
\left(\frac1{|x-x_0|^2} - \frac1{r^2}\right) dS(x) 
+ \frac1{2\pi^2 r^3} \int_{\O \cap \dd B_r (x_0)} u(x) \, dS(x) .
\nonumber
\end{eqnarray}
The first term in the right hand side is not greater than 
$\int_{\O \cap B_r(x_0)} \frac{\ph (x) \, dx}{4 \pi^2 |x-x_0|^2}$.
By virtue of Lemma \ref{l43} $u(x) \ge 0$ in $\O$.
Therefore, $\left.\frac{\dd u}{\dd\nu}\right|_{\dd\O} \le 0$, 
and the second term in the right hand side is non-positive,
as $|x-x_0|^2 \le r^2$ for $x \in B_r(x_0)$.
Now, \eqref{43} follows from \eqref{46}.
\end{proof}

\begin{cor}
\label{c455}
Let $\ph \in L_2(\O) \cap C^\ga (\overline\O)$, $\ga>0$, and $\ph(x) \ge 0$ for all $x \in \O$.
Let $u \in \mathring W_2^1 (\O)$ be the solution to the Dirichlet problem \eqref{415}.
Let $x_0 = (w_0, y_0)$, $0<R\le 1$.
Then 
\begin{equation*}
u (x_0) \le \int_{\O \cap B_R(x_0)} \frac{\ph (x) \, dx}{4 \pi^2 |x-x_0|^2} 
+ \frac2{\pi^2 R^4} \int_{\O \cap B_R(x_0)} u(x) \, dx.
\end{equation*}
\end{cor}

\begin{proof}
For almost all $r \in (0,R)$ we have
$$
2 \pi^2 r^3 u (x_0) \le 2 \pi^2 r^3 \int_{\O \cap B_R(x_0)} \frac{\ph (x) \, dx}{4 \pi^2 |x-x_0|^2} 
+ \int_{\O \cap \dd B_r(x_0)} u(x) \, dS(x).
$$
Integrating this inequality with respect to $r$ from $0$ to $R$ we get
$$
\frac{\pi^2R^4}2 u(x_0) 
\le \frac{\pi^2R^4}2 \int_{\O \cap B_R(x_0)} \frac{\ph (x) \, dx}{4 \pi^2 |x-x_0|^2} 
+ \int_{\O \cap B_R(x_0)} u(x) \, dx.
\qquad \qedhere
$$
\end{proof}

\begin{lemma}
\label{l46}
Let 
$$
\ph \in L_2(\O) \cap L_\infty(\O) \cap C^\ga (\overline\O), \quad \ga>0, 
\qquad \text{and} \quad \ph(x) \ge 0 \quad \forall\ x \in \O.
$$
Let $u \in \mathring W_2^1 (\O)$ be the solution to the Dirichlet problem \eqref{415}.
Then $u \in L_\infty (\O)$ and
$$ 
\|u\|_{L_\infty (\O)} 
\le \frac14 \|\ph\|_{L_\infty(\O)} + \|\ph\|_{L_2(\O)}.
$$
\end{lemma}

\begin{proof}
Let $x_0 = (w_0, y_0) \in \O$.
By Corollary \ref{c455} with $R=1$ we have
\begin{equation}
\label{439}
u (x_0) \le \int_{\O \cap B_1(x_0)} \frac{\ph (x) \, dx}{4 \pi^2 |x-x_0|^2} 
+ \frac2{\pi^2} \int_{\O \cap B_1(x_0)} u(x) \, dx.
\end{equation}
The first term in the right hand side of \eqref{439} is not greater than
$$
\int_{\O \cap B_1(x_0)} \frac{\ph (x) \, dx}{4 \pi^2 |x-x_0|^2} 
\le \frac{\|\ph\|_{L_\infty(\O)}}{4 \pi^2} \int_{B_1(x_0)} \frac{dx}{|x-x_0|^2} 
= \frac{\|\ph\|_{L_\infty(\O)}}4 .
$$
For the second term in the right hand side of \eqref{439} we have due to \eqref{435}
$$
\frac2{\pi^2} \int_{\O \cap B_1(x_0)} u(x) \, dx
\le \frac2{\pi^2} \left(\frac{\pi^2}2\right)^{1/2} \|u\|_{L_2(\O\cap B_1(x_0))} 
\le \frac{\sqrt 2}\pi \cdot \frac{15}7 \|\ph\|_{L_2(\O)} \le \|\ph\|_{L_2(\O)}.
$$
Thus,
$$
0 \le u(x_0) \le \frac14 \|\ph\|_{L_\infty(\O)} + \|\ph\|_{L_2(\O)}.
\quad \qedhere
$$
\end{proof}

%%%%%%%%%%%%%%%%%%%%%%%%%%%%%%%%%%%%%%%%%%
\section{Dirichlet problem in $\O$ with exponentially decaying RHS}
Introduce the following notations
$$
R(w) = \begin{cases}
50 \sqrt h \,w_m^{-1/3}, & \quad \text{if} \ |w| < w_m, \\
50 \sqrt h \,|w|^{-1/3}, & \quad \text{if} \ |w| \ge w_m, \end{cases}
\qquad B(x) : = B_{R(w)} (x), \quad x = (w,y).
$$

\begin{lemma}
\label{l511}
For all $w \ge 0$ we have
$$
R\left(w-R(w)\right) \le \frac{10}9 \,R(w) 
$$
and
\begin{equation}
\label{48}
\al(w)^{4/3} - \al\left(w-R(w)\right)^{4/3} \le 200 \sqrt h .
\end{equation}
\end{lemma}

\begin{proof}
If $0 \le w \le w_m$ then $R(w-R(w)) = R(w)$.

If $w \ge w_m$ then 
$w - R(w) \ge \left(\frac9{10}\right)^3 w$, and therefore,
$$
R\left(w-R(w)\right) \le \frac{50 \sqrt h}{(w-R(w))^{1/3}} 
\le \frac{500 \sqrt h}{9 \,w^{1/3}} = \frac{10}9 \,R(w) .
$$

For the second inequality we have
due to Corollary \ref{c145}
$$
\al(w)^{4/3} - \al\left(w-R(w)\right)^{4/3} 
\le 4 \,\al(w)^{1/3} R(w) \le 200 \sqrt h . \quad \qedhere
$$
\end{proof}

Assume that a given smooth function $\ph$ in $\O$ satisfies the inequality
\begin{equation}
\label{ph}
0 \le \ph(x) \le 10 \,e^{-\al(w)^{4/3}} .
\end{equation}

\begin{lemma}
\label{lPh}
If a function $\ph$ satisfies \eqref{ph} then the function
\begin{equation}
\label{Ph}
\Phi (x) = \int_{\O\cap B(x)} \frac{\ph(\tilde x)\, d\tilde x}{4 \pi^2 |x-\tilde x|^2}
\end{equation}
satisfies the estimate 
$$
\Phi(x) \le 3 R(w)^2 e^{-\al(w)^{4/3}} .
$$
\end{lemma}

\begin{proof}
It is enough to prove the estimate for $w \ge 0$.
By assumption
$$
\Phi(x) \le 
10\, e^{-\al(w-R(w))^{4/3}} \int_{B(x)} \frac{d\tilde x}{4 \pi^2 |x-\tilde x|^2}.
$$
We have also
$$
\int_{B(x)} \frac{d\tilde x}{4 \pi^2 |x-\tilde x|^2} = \frac{R(w)^2}4.
$$
Finally, by virtue of the preceding Lemma
$$
e^{-\al(w-R(w))^{4/3}} \le e^{200 \sqrt h} e^{-\al(w)^{4/3}} .
$$
The claim follows.
\end{proof}

Let the function $\Phi$ be defined by \eqref{Ph}
with $\ph$ satisfying \eqref{ph}.
In the space $L_\infty (\O)$ we consider the map $T$ defined by the formula
\begin{equation}
\label{T}
(T\rho) (x) = \Phi (x) + \frac2{\pi^2 R(w)^4} \int_{\O \cap B(x)} \rho(\tilde x) \, d \tilde x.
\end{equation}
Clearly,
\begin{eqnarray*}
\left|(T\rho_1)(x) - (T\rho_2)(x)\right| 
= \left|\frac2{\pi^2 R(w)^4} 
\int_{\O \cap B(x)} \left(\rho_1(\tilde x) - \rho_2(\tilde x)\right) d\tilde x\right|\\
\le \frac{\mes_4 \left(\O \cap B(x)\right)}{\mes_4 B(x)} \,\|\rho_1-\rho_2\|_{L_\infty(\O)}
\le \frac12 \,\|\rho_1-\rho_2\|_{L_\infty(\O)}
\end{eqnarray*}
due to Corollary \ref{c261}.
Thus, the map $T$ is a contraction, and hence, there exists a unique function
$\rho_* \in L_\infty (\O)$ such that $T \rho_* = \rho_*$.

\begin{lemma}
\label{lrho}
a) If $u \in L_\infty (\O)$ satisfies the inequality
$$
u(x) \le (Tu) (x) \qquad \text{a.e. in } \O,
$$
then $u(x) \le \rho_* (x)$ a.e. in $\O$.

b)  If $u \in L_\infty (\O)$ satisfies the inequality
$$
u(x) \ge (Tu) (x) \qquad \text{a.e. in } \O,
$$
then $u(x) \ge \rho_* (x)$ a.e. in $\O$.
\end{lemma}

\begin{proof}
a) Let us consider a function 
$$
\rho_0(x) = \const > \|u\|_{L_\infty(\O)}.
$$
Put $\rho_k = T^k \rho_0$.
Clearly, 
$$
\rho_k \mathop{\longrightarrow}\limits_{k \to \infty} \rho_* \quad \text{in} \quad L_\infty (\O).
$$
On the other hand, we have
$$
u(x) \le \rho_k (x) \qquad \text{a.e. in } \O
$$
by induction in $k$.
Taking a limit $k \to \infty$ we get $u(x) \le \rho_* (x)$.

b) Now, take a function
$$
\rho_0(x) = \const < - \|u\|_{L_\infty(\O)},
$$
and $\rho_k = T^k \rho_0$.
In the same way,
$$
u(x) \ge \rho_k (x) \qquad \text{a.e. in } \O,
$$
and therefore, $u(x) \ge \rho_* (x)$ a. e. in $\O$.
\end{proof}

\begin{lemma}
\label{lPsi}
Put
$$
\Psi (x) = 10 R(w)^2 e^{-\al(w)^{4/3}}.
$$
Then
$$
\frac2{\pi^2 R(w)^4} \int_{\O \cap B(x)} \Psi (\tilde x) \, d \tilde x
\le \frac7{10} \, \Psi (x).
$$
\end{lemma}

\begin{proof}
It is enough to consdier the case $w \ge 0$.
We have
\begin{eqnarray*}
\frac2{\pi^2 R(w)^4} \int_{\O \cap B(x)} \Psi (\tilde x) \, d \tilde x\\
\le \frac2{\pi^2 R(w)^4} \cdot 10 \, R(w-R(w))^2 e^{-\al(w-R(w))^{4/3}} 
\cdot \mes_4 \left(\O \cap B(x)\right) \\
\le \frac12 \cdot 10 \cdot \frac{100}{81} R(w)^2 e^{200\sqrt h} e^{-\al(w)^{4/3}}
\end{eqnarray*}
due to Corollary \ref{c261} and Lemma \ref{l511}.
\end{proof}

\begin{theorem}
\label{t48}
Let 
$$
\ph \in C^\ga (\overline\O), \quad \ga > 0, \qquad
0 \le \ph(x) \le 10 \, e^{-\al(w)^{4/3}} \quad \forall \ x \in \O.
$$
Let $u$ be the solution to the Dirichlet problem
$$
\begin{cases}
- \D u (x) = \ph (x) \quad \text{in} \ \O, \\
\left.u\right|_{\dd\O} = 0.
\end{cases}
$$
Then 
$$
0 \le u(x) \le 10 \,R(w)^2 e^{-\al(w)^{4/3}} \quad \forall \ x \in \O.  
$$
\end{theorem}

\begin{proof}
Let $\Phi$ be the function defined by \eqref{Ph} with the given function $\ph$,
and $T$ be the map defined by \eqref{T}.
By Corollary \ref{c455} $u(x) \le (Tu)(x)$ for all $x \in \O$.
On the other hand, for the function $\Psi$ defined in the preceding Lemma 
we have due to Lemma \ref{lPh} and Lemma \ref{lPsi}
$$
(T\Psi) (x) = \Phi (x) + \frac2{\pi^2 R(w)^4} \int_{\O \cap B(x)} \Psi (\tilde x) \, d \tilde x
\le 3\,R(w)^2 e^{-\al(w)^{4/3}} + \frac7{10} \, \Psi (x) = \Psi (x). 
$$
Now, Lemma \ref{lrho} implies
$$
u(x) \le \rho_* (x) \le \Psi (x). \qquad \qedhere
$$
\end{proof}

\begin{cor}
\label{c49}
Let $u \in \mathring W_2^1 (\O)$ be the solution to the Dirichlet problem
$$
\begin{cases}
- \D u (x) = \D f (x) \quad \text{in} \ \O, \\
\left.u\right|_{\dd\O} = 0.
\end{cases}
$$
with the function $f$ defined in \S \ref{s1}.
Then
$$
|u(x)| \le 10 \,R(w)^2 e^{-\al(w)^{4/3}} \quad \forall \ x \in \O.  
$$
\end{cor}

\begin{proof}
For $x \in \O$ we have the estimate
$|\D f(x)| \le 10 e^{-\al(w)^{4/3}}$.
Indeed, if $|w| \ge w_{m+1}$ then it follows directly from Lemma \ref{l13}.
If $|w| \le w_{m+1}$ then by Lemma \ref{l13}
$$
|\D f(x)| \le 40 m |f(x)| \le 80 m h e^{-\al(w)^{4/3}} < 10 e^{-\al(w)^{4/3}}
$$
by definition of the numbers $m$ and $h$.

Now, denote
$$
(\D f)_+ (x) = \max (\D f(x), 0), \qquad
(\D f)_- (x) = \max (0, - \D f(x)).
$$
Then
$$
\D f = (\D f)_+ - (\D f)_-, \qquad
(\D f)_\pm (x) \ge 0 \quad \forall \ x.
$$
By Lemma \ref{l13} $\D f \in C^\infty (\Pi)$,
therefore, $(\D f)_\pm \in C^\ga (\Pi)$ for all $\ga < 1$.

Denote by $u^{(\pm)} \in \mathring W_2^1 (\O)$ the solutions to the Dirichlet problems
for the equations $- \D u^{(\pm)} = (\D f)_\pm$.
By Theorem \ref{t48}
$$
0 \le u^{(\pm)} (x) \le 10 \, R(w)^2 e^{-\al(w)^{4/3}}.
$$
The function $u^{(+)} - u^{(-)}$ is a solution to the Dirichlet problem 
for the equation $- \D u = \D f$.
Due to uniqueness of the solution we have $u = u^{(+)} - u^{(-)}$.
Therefore,
$|u(x)| \le 10 R(w)^2 e^{-\al(w)^{4/3}}$.
\end{proof}

\begin{rem}
The function $u$ is even, $u((w,y)) = u((-w,y))$.
\end{rem}

%%%%%%%%%%%%%%%%%%%%%%%%%%%%%%%%%%%%%%%%%%
\section{Function F}
\subsection{Construction of the function $F$}
Define in $\O$ the function
$$
g(x) = f(x) + u(x),
$$
where the functions $f$ and $u$ are defined in \S \ref{s1} and in Corollary \ref{c49}
respectively.
The function $g$ solves the problem
$$
\begin{cases}
\D g(x) = 0 \quad \text{in} \ \O, \\
\left.g\right|_{\dd\O} = \left.f\right|_{\dd\O} .
\end{cases}
$$
In the whole cylinder $\Pi$ we define the function 
$$
f^* (x) = \begin{cases}
f(x), \quad & \text{if} \ x \notin \O, \\
g(x), \quad & \text{if} \ x \in \O.
\end{cases}
$$
Clearly, $f^* \in C(\Pi)$ and
\begin{equation}
\label{51}
|f^*(x)| \le 2 \,e^{-\al(w)^{4/3}} \qquad \forall \ x \in \Pi.
\end{equation}

Introduce the radius
$$
\rho : = \frac{h}4
$$
where $h$ is defined by \eqref{h}.
Fix a function $\psi \in C_0^\infty (B_\rho)$, $B_\rho \subset \R^4$, 
such that $\int_{B_\rho} \psi (x) \, dx = 1$, $\psi(x) \ge 0$ everywhere, 
and this function is radially symmetric, $\psi (x) = \hat\psi (|x|)$.
Noting that
\begin{equation}
\label{52}
\al(w)^{4/3} \int_{\R^4} \psi \left((x-\tilde x) \al(w)^{1/3}\right) \, d \tilde x = 1,
\end{equation}
we define the function
$$
F(x) := \al(w)^{4/3} \int_\Pi \psi \left((x-\tilde x) \al(w)^{1/3}\right) f^* (\tilde x) \, d \tilde x.
$$
Clearly, $F \in C^\infty (\Pi)$.
It is also clear that the functions $g$, $f^*$ and $F$ are even.

\begin{lemma}
\label{l51}
We have
$$
|F(x)| \le 3 \,e^{-\al(w)^{4/3}}.
$$
\end{lemma}

\begin{proof}
It follows from \eqref{51}, \eqref{48} and \eqref{52}.
\end{proof}

\begin{lemma}
\label{l52}
If $|x-\tilde x| \le \rho\, \al(w)^{-1/3}$ then $|f_0(x) - f_0(\tilde x)| \le h/2$.
\end{lemma}

\begin{proof}
Without loss of generality one can assume $\tilde w \ge 0$.
If $w \le w_m$ then $\al(w) = w_m$, $\tilde w \le w_{m+1}$,
and due to Lemma \ref{l12}
$$
|f_0(x) - f_0(\tilde x)| \le \sqrt{4m+5} \,|y-\tilde y|  
\le \sqrt{4m+5} \, \rho\, w_m^{-1/3} = \frac{\sqrt{4m+5} \, h}{6 \sqrt m} < \frac{h}2.
$$

If $w_k \le w \le w_{k+1}$, $k \ge m$, then 
$\al(w) \ge w_k$, $\tilde w \in [w_{k-1}, w_{k+2}]$ and by virtue of Lemma \ref{l12}
$$
|f_0(x)-f_0(\tilde x)| \le \sqrt{5(k+1)} \,|x-\tilde x|  
\le \sqrt{5(k+1)}\,\rho\, w_k^{-1/3} = \frac{\sqrt{5(k+1)} \, h}{6 \sqrt k} < \frac{h}2.
\qquad \qedhere
$$
\end{proof}

We consider also the set
$$
U := \left\{x \in\Pi : |f_0(x)| > h\right\}.
$$
Clearly,
$$
U \cup \O = \Pi.
$$

\begin{lemma}
\label{l53}
If $x \in U$ then
$$
|F(x)| \ge \frac{h}8\, e^{-\al(w)^{4/3}}.
$$
\end{lemma}

\begin{proof}
Once again, it is sufficient to consider $w \ge 0$.
Assume first that $f_0(x) > h$.
Consider the points $\tilde x \in B_{\rho \al(w)^{-1/3}} (x)$.
By the preceding Lemma
$$
f_0 (\tilde x) > \frac{h}2 .
$$
If $\tilde x \notin \O$ then 
$f_0(\tilde x) > 2h$, and $f^* (\tilde x) = f(\tilde x) > 2 h e^{-\al(\tilde w)^{4/3}}$.
If $\tilde x \in \O$ then 
$$
f^* (\tilde x) = g(\tilde x) \ge f(\tilde x) - |u(\tilde x)| 
\ge \left(\frac{h}2 - 10 R(\tilde w)^2\right) e^{-\al(\tilde w)^{4/3}} 
\ge \frac{h}6 \,e^{-\al(\tilde w)^{4/3}}
$$
due to Corollary \ref{c49}.
In both cases 
$$
f^*(\tilde x) \ge \frac{h}6 e^{-\al(\tilde w)^{4/3}}.
$$
Therefore,
$$
F(x) \ge \frac{h}6 \,e^{-\al(w + \rho \al(w)^{-1/3})^{4/3}} \ge \frac{h}8 \,e^{-\al(w)^{4/3}},
$$
because due to Corollary \ref{c145}
\begin{eqnarray*}
\al\left(w + \rho \al(w)^{-1/3}\right)^{4/3} - \al(w)^{4/3} 
%\le \left(\al(w) + 3\rho \al(w)^{-1/3}\right)^{4/3} - \al(w)^{4/3} \\
\le 4 \rho \al(w)^{-1/3} \left(\al(w) + 3\rho \al(w)^{-1/3}\right)^{1/3} \\
= 4 \rho \left(1 + 3 \rho \al(w)^{-4/3}\right)^{1/3} 
\le 4 \rho \left(1 + 3 \rho w_m^{-4/3}\right)^{1/3} < \ln \frac43.
\end{eqnarray*}
In the case $f_0(x) < - h$ a similar argument shows that
$$
F(x) \le - \frac{h}8 \,e^{-\al(w)^{4/3}}.
\qquad \qedhere
$$
\end{proof}

\begin{lemma}
\label{l54}
If $x \notin U$ then $F(x) = g(x)$.
\end{lemma}

\begin{proof}
Let $x \notin U$. 
By definition $|f_0(x)| \le h$.
By Lemma \ref{l52} $|f_0(\tilde x)| \le \frac{3h}2$ if $|\tilde x - x| \le \rho \al(w)^{-1/3}$, 
thus 
$$
B_{\rho \al(w)^{-1/3}} (x) \subset \O.
$$
Therefore,
$$
F(x) =  \al(w)^{4/3} 
\int_{B_{\rho \al(w)^{-1/3}}(x)} \psi \left((x-\tilde x) \al(w)^{1/3}\right) g (\tilde x) \, d \tilde x.
$$
In spherical coordinates
$$
\tilde x - x = (t; \te), \quad t \in [0, \infty), \ \te \in S^3,
$$
we have 
$$
F(x) =  \al(w)^{4/3} \int_0^{\rho \al(w)^{-1/3}} \int_{S^3} 
\hat \psi \left(t \al(w)^{1/3}\right) g \left(x+ (t;\te)\right) t^3 dt \, dS(\te).
$$
The function $g$ is harmonic in $\O$, 
so by the mean value property
$$
\int_{S^3} g \left(x+ (t;\te)\right) \, dS(\te) = 2 \pi^2 g(x).
$$
Moreover,
$$
2 \pi^2 \int_0^{\rho \al(w)^{-1/3}} \hat \psi \left(t \al(w)^{1/3}\right) t^3 dt 
= \int_{|z| \le \rho \al(w)^{-1/3}} \psi \left(z \al(w)^{1/3}\right) dz = \al(w)^{-4/3}.
$$
Therefore, $F(x) = g(x)$.
\end{proof}

\begin{cor}
\label{c55}
If $x \notin U$ then $\D F(x) = 0$.
\end{cor}

\subsection{Laplacian of $F$}
\begin{lemma}
\label{l55}
We have
\begin{equation}
\label{D1}
\left|\D_x \left(\al(w)^{4/3} \psi \left((x-\tilde x) \al(w)^{1/3}\right)\right) 
- \D_{\tilde x} \left(\al(w)^{4/3} \psi \left((x-\tilde x) \al(w)^{1/3}\right)\right)\right| 
\le C_1\, \al(w)^{2/3}
\end{equation}
and
\begin{equation}
\label{D2}
\left|\D_x \left(\al(w)^{4/3} \psi \left((x-\tilde x) \al(w)^{1/3}\right)\right)\right| \le C_1\, \al(w)^2.
\end{equation}
Here $C_1$ is a constant that depends on the choice of the function $\psi$ only.
\end{lemma}

\begin{proof}
Clearly,
\begin{eqnarray}
\nonumber
\D_y \left(\al(w)^{4/3} \psi \left((x-\tilde x)\al(w)^{1/3}\right)\right) \\
= \al(w)^2 (\D_y \psi) \left((x-\tilde x) \al(w)^{1/3}\right)
\label{53} \\
= \D_{\tilde y} \left(\al(w)^{4/3} \psi \left((x-\tilde x) \al(w)^{1/3}\right)\right)
\nonumber
\end{eqnarray}
and
\begin{equation}
\label{54}
\dd_{\tilde w}^2 \left(\al(w)^{4/3} \psi \left((x-\tilde x) \al(w)^{1/3}\right)\right)
= \al(w)^2 (\dd_w^2 \psi) \left((x-\tilde x) \al(w)^{1/3}\right).
\end{equation}
We have also
\begin{eqnarray*}
\dd_w \left(\al(w)^{4/3} \psi \left((x-\tilde x) \al(w)^{1/3}\right)\right) \\
= \frac43 \, \al(w)^{1/3} \al'(w) \psi \left((x-\tilde x) \al(w)^{1/3}\right) \\
+ \al(w)^{4/3} \left(\al(w)^{1/3} + \frac13\, (w-\tilde w) \al(w)^{-2/3} \al'(w)\right) 
(\dd_w \psi) \left((x-\tilde x) \al(w)^{1/3}\right) \\
+ \frac13 \, \al(w)^{2/3} \al'(w)
\sum_{j=1}^3 (y_j - \tilde y_j) (\dd_{y_j} \psi) \left((x-\tilde x) \al(w)^{1/3}\right)
\end{eqnarray*}
and
\begin{eqnarray}
\nonumber
\dd_w^2 \left(\al(w)^{4/3} \psi \left((x-\tilde x) \al(w)^{1/3}\right)\right) \\
= \left(\al(w) + \frac13\, (w-\tilde w) \al'(w)\right)^2 
(\dd_w^2 \psi) \left((x-\tilde x) \al(w)^{1/3}\right) + O (\al(w)^{2/3}),
\label{55}
\end{eqnarray}
where we have taken into account the relations
$$
|x-\tilde x| \le \rho \al(w)^{-1/3}, \quad \tilde w = O(\al(w)),
\quad \psi, \n \psi, \n^2 \psi = O (1), 
$$
$$
\quad \al'(w) = O(1), \quad \al''(w) = O \left(\al(w)^{-1/3}\right).
$$
Further,
$$
\left(\al(w) + \frac13\, (w-\tilde w) \al'(w)\right)^2 - \al(w)^2 = O (\al(w)^{2/3}).
$$
Now, the result follows from \eqref{53}, \eqref{54} and \eqref{55}.
\end{proof}

Let us consider also the function
$$
G(x) := \al(w)^{4/3} \int_\Pi \psi \left((x-\tilde x) \al(w)^{1/3}\right) f (\tilde x) \, d \tilde x.
$$
Clearly, $G \in C^\infty (\Pi)$.

\begin{lemma}
\label{l56}
We have
$$
|\D G(x)| \le C_2 \,e^{-\al(w)^{4/3}},
$$
where the constant $C_2$ depends on the choice of the function $\psi$ only.
\end{lemma}

\begin{proof}
It is sufficient to consider $w\ge 0$.
Represent $\D G$ in the form
\begin{eqnarray*}
\D G(x) = 
\int\left(\D_x \left(\al(w)^{4/3} \psi \left((x-\tilde x) \al(w)^{1/3}\right)\right) 
- \D_{\tilde x} \left(\al(w)^{4/3} \psi \left((x-\tilde x) \al(w)^{1/3}\right)\right)\right) 
f(\tilde x) \, d\tilde x\\
+ \int \D_{\tilde x} \left(\al(w)^{4/3} \psi \left((x-\tilde x) \al(w)^{1/3}\right)\right) 
f(\tilde x) \, d\tilde x =: I_1 + I_2.
\end{eqnarray*}
Estimate the first integral using \eqref{D1}:
\begin{eqnarray*}
|I_1(x)| \le
C_1 \al(w)^{2/3} \int_{|\tilde x - x| \le \rho \al(w)^{-1/3}} |f(\tilde x)| \, d \tilde x \\
\le C_1 \al(w)^{2/3} e^{- \al(w-\rho \al(w)^{-1/3})^{4/3}} \,\frac{\pi^2 \rho^4 \al(w)^{-4/3}}2.
%\le C_1 \pi^2 \rho^4  \al(w)^{-2/3} e^{-\al(w)^{4/3}}.
\end{eqnarray*}
By Corollary \ref{c145}
\begin{equation}
\label{al2}
\al(w)^{4/3} - \al(w-\rho \al(w)^{-1/3})^{4/3} 
\le 4 \rho = h.
\end{equation}
Therefore,
$$
|I_1(x)| \le \frac12 C_1 \pi^2 \rho^4  \al(w)^{-2/3} e^h e^{-\al(w)^{4/3}}.
$$

Next, integrating by parts we obtain
$$
I_2(x) = 
\int \al(w)^{4/3} \psi \left((x-\tilde x) \al(w)^{1/3}\right) \D f(\tilde x) \, d\tilde x.
$$
Therefore, by virtue of Corollary \ref{c155}, \eqref{52} and \eqref{al2}
$$
|I_2(x)| \le 40 m e^{- \al(w-\rho \al(w)^{-1/3})^{4/3}}
\le 40 m e^h e^{-\al(w)^{4/3}}. 
\qquad \qedhere
$$
\end{proof}

\begin{lemma}
\label{l57}
We have
$$
|\D F(x)| \le C_3 \,e^{-\al(w)^{4/3}} ,
$$
where the constant $C_3$ depends on the choice of the function $\psi$ only.
\end{lemma}

\begin{proof}
Once again we consider $w \ge 0$.
We have
$$
\D F(x) - \D G(x) = 
\int_\Pi \D_x \left(\al(w)^{4/3} \psi \left((x-\tilde x) \al(w)^{1/3}\right)\right)
\left(f^* (\tilde x) - f (\tilde x)\right) d \tilde x.
$$
If $\tilde x \notin \O$ then $f^* (\tilde x) - f (\tilde x) = 0$.
If $\tilde x \in \O$ then
$$
\left|f^* (\tilde x) - f (\tilde x)\right| = |u(\tilde x)| \le 10 R(\tilde w)^2 e^{-\al(\tilde w)^{4/3}}
$$
due to Corollary \ref{c49}.
Taking into account \eqref{D2} and \eqref{al2} we get
\begin{eqnarray}
\nonumber
\left|\D F(x) - \D G(x)\right| \le 
10 C_1 \al(w)^2 \int_{|\tilde x - x| \le \rho \al(w)^{-1/3}} 
R(\tilde w)^2 e^{-\al(\tilde w)^{4/3}} d\tilde x \\
\label{last}
\le 10 C_1 \al(w)^2 R(w-\rho \al(w)^{-1/3})^2 e^{-\al(w-\rho \al(w)^{-1/3})^{4/3}}
\frac{\pi^2 \rho^4 \al(w)^{-4/3}}2 \\
\le 5 C_1 \al(w)^{2/3} R(w-\rho \al(w)^{-1/3})^2 \pi^2 \rho^4 e^h e^{-\al(w)^{4/3}}.
\nonumber
\end{eqnarray}
Show that
\begin{equation}
\label{lastlast}
\al(w)^{1/3} R(w-\rho \al(w)^{-1/3}) \le 1.
\end{equation}
Indeed, if $w \le w_m$ then 
$$
\al(w)^{1/3} R(w-\rho \al(w)^{-1/3}) = 50 \sqrt h.
$$
If $w_m \le w \le w_{m+1}$, then 
$$
\al(w)^{1/3} R(w-\rho \al(w)^{-1/3}) \le 50 \sqrt h \,w_m^{-1/3} w_{m+1}^{1/3}.
$$
If $w_k \le w \le w_{k+1}$, $k \ge m+1$ then
$$
\al(w)^{1/3} R(w-\rho \al(w)^{-1/3}) \le 50 \sqrt h\, w_{k-1}^{-1/3} w_{k+1}^{1/3}.
$$
In all cases \eqref{lastlast} follows.
Now, the claim follows from \eqref{last}, \eqref{lastlast} and Lemma \ref{l56}.
\end{proof}

{\it Proof of Theorem \ref{t04}.}
The existence of such function $F$ 
follows from Lemma \ref{l51}, Lemma \ref{l53},
Corollary \ref{c55} and Lemma \ref{l57}.
\qed

%%%%%%%%%%%%%%%%%%%%%%%%%%%%%%%%%%%%%%%%%%

\end{document}